\documentclass[12pt,a4paper]{amsart}
\usepackage{latexsym}
\usepackage{amsmath,amssymb,amsfonts,amsthm,amscd}
\usepackage[all,arc]{xy}
\usepackage{graphicx}
\usepackage{enumerate}
\usepackage{mathrsfs}
\usepackage{comment}
\usepackage{arydshln}
\usepackage{mathtools}
\usepackage{colonequals}
\usepackage{braket}

\newtheorem{thm}{Theorem}[section]
\newtheorem{cor}[thm]{Corollary}
\newtheorem{prop}[thm]{Proposition}
\newtheorem{lem}[thm]{Lemma}

\newtheorem*{main}{Main Theorem}

\theoremstyle{definition}
\newtheorem{defi}[thm]{Definition}

\theoremstyle{remark}

\newcommand{\cf}{\textit{cf.\ }}

\newcommand{\bA}{\mathbb{A}}
\newcommand{\bC}{\mathbb{C}}

\newcommand{\bF}{\mathbb{F}}

\newcommand{\bN}{\mathbb{N}}

\newcommand{\bQ}{\mathbb{Q}}
\newcommand{\bR}{\mathbb{R}}
\newcommand{\bS}{\mathbb{S}}

\newcommand{\bZ}{\mathbb{Z}}

\newcommand{\cC}{\mathcal{C}}
\newcommand{\cD}{\mathcal{D}}

\newcommand{\cG}{\mathcal{G}}

\newcommand{\cL}{\mathcal{L}}
\newcommand{\cM}{\mathcal{M}}
\newcommand{\cN}{\mathcal{N}}
\newcommand{\cO}{\mathcal{O}}
\newcommand{\cP}{\mathcal{P}}
\newcommand{\cQ}{\mathcal{Q}}
\newcommand{\cR}{\mathcal{R}}

\newcommand{\fa}{\mathfrak{a}}

\newcommand{\fm}{\mathfrak{m}}

\newcommand{\fp}{\mathfrak{p}}

\newcommand{\fS}{\mathfrak{S}}

\newcommand{\fX}{\mathfrak{X}}

\newcommand{\sD}{\mathscr{D}}

\newcommand{\sG}{\mathscr{G}}

\newcommand{\sS}{\mathscr{S}}

\newcommand{\ol}{\overline}
\newcommand{\ul}{\underline}
\newcommand{\wt}{\widetilde}
\newcommand{\wh}{\widehat}

\newcommand{\loc}{\mathrm{loc}}
\newcommand{\alg}{\mathrm{alg}}

\newcommand{\Gm}{\mathbb{G}_{\mathrm{m}}}

\newcommand{\Mloc}{\mathrm{M}^{\loc}}

\newcommand{\isoarrow}{\xrightarrow\sim}

\DeclareMathOperator{\Spec}{Spec}
\DeclareMathOperator{\Spf}{Spf}
\DeclareMathOperator{\Hom}{Hom}
\DeclareMathOperator{\Isom}{Isom}
\DeclareMathOperator{\HermIsom}{HermIsom}

\DeclareMathOperator{\End}{End}
\DeclareMathOperator{\Stab}{Stab}
\DeclareMathOperator{\Ker}{Ker}

\DeclareMathOperator{\Res}{Res}
\DeclareMathOperator{\Nm}{Nm}
\DeclareMathOperator{\Tr}{Tr}
\DeclareMathOperator{\Lie}{Lie}
\DeclareMathOperator{\Sh}{Sh}
\DeclareMathOperator{\GL}{GL}
\DeclareMathOperator{\GSp}{GSp}
\DeclareMathOperator{\GU}{GU}
\DeclareMathOperator{\oU}{U}
\DeclareMathOperator{\Gr}{Gr}
\DeclareMathOperator{\Def}{Def}
\DeclareMathOperator{\diag}{diag}
\DeclareMathOperator{\inv}{inv}

\makeatletter
\let\c@equation\c@thm
\makeatother
\numberwithin{equation}{section}

\title{The canonicity of the integral models of RSZ Shimura varieties}
\author{Yuta Nakayama}
\date{}
\begin{document}

\begin{abstract}
We show that the integral models of Shimura varieties of Rapoport, Smithling and Zhang in relation to variants of the arithmetic Gan--Gross--Prasad conjecture, the arithmetic fundamental lemma conjecture and the arithmetic transfer conjecture are canonical in the sense prescribed by Pappas. In particular, we prove that they are isomorphic to the models constructed by Kisin and Pappas.
\end{abstract}

\maketitle
\setcounter{tocdepth}{1}
\tableofcontents

\section{Introduction} \label{sec:Intro}

The Shimura varieties have long been important in number theory because of their links to both arithmetic geometry and automorphic representation theory.
The Langlands program is a famous example of such links, and another emerging subject is the intersection theory of special cycles on these varieties and its relation to automorphic $L$-functions, starting from the Gross--Zagier formula \cite{GZ}.

In this direction, Rapoport, Smithling and Zhang \cite{RSZInt} have recently formulated a variant of the arithmetic Gan--Gross--Prasad conjecture and global analogues of the arithmetic fundamental lemma conjecture and the arithmetic transfer conjecture (\cf\cite[\S 27]{GGPConj}, \cite{ZhaAFLDef} and \cite{RSZATDef}).
One of the main parts of the formulation is their explicit moduli theoretic construction of semi-global and global integral models of certain Shimura varieties.
The sequel paper \cite{RSZExp} has vastly generalized the construction, and is the subject of our paper.
Their moduli problem is about pairs of abelian schemes with additional data unlike classical integral models in \cite{KotInt} and \cite{RZ}.

As the last two references show, the integral models of Shimura varieties are of arithmetic interest in their own right, and have been investigated by various people. Lately, Kisin and Pappas \cite{KPInt} constructed well-behaved semi-global integral models for a large class of Shimura varieties with parahoric level structure at $p$.
These models apparently depend on various choices, but Pappas \cite{PapInt} has introduced the notion of canonical integral models in a broader context than \cite{MilInt} and its analogues in the ramified case, and has showed that the models of Shimura varieties of Hodge type in \cite[4.2.1]{KPInt} are independent of choices by the fact that they are canonical.
However, such models are abstractly defined, not easy to study for some applications.

Our goal in this paper is to show that the RSZ models are also canonical.
This also implies that models in \cite[Theorem 5.4]{RSZExp} coincide with those in \cite[4.2.1]{KPInt}.
This paper has potential effects on future research in two directions.
First, since the appearance of \cite{PapInt}, there have been several papers which refer to some integral models being canonical, e.g., \cite[p.~501]{HPRRed}, \cite[p.~3]{ZacLoc} and \cite[p.~3]{PZInt}, but there are still few research on this canonicity.
Second, \cite{RSZInt} has posed the question of formulating a variant of the arithmetic Gan--Gross--Prasad conjecture using the Kisin--Pappas integral models.
Our paper suggests that this question is reasonable.

In fact, we could have done without the canonicity by showing that morphisms from the RSZ integral models to the integral models of the Siegel modular varieties in \S \ref{sec:LocUniv-MorKP} are proper and quasi-finite.
\cite[Lemma B.2.5, B.2.28]{QiuMod} in a paper that appeared later than this paper also show similar results to ours based on the explicit construction of the Kisin--Pappas integral models.
However, we adopt our method of using the canonicity in search for its other uses than showing that the Kisin--Pappas integral models for Hodge type Shimura varieties are independent of the choice of Hodge embeddings.

The condition for a system of models of Shimura varieties to be canonical consists mainly of two parts.
One is the extension property close to \cite[Definition 2.5]{MilInt}, which we cope with in a similar way to the exisiting method.
The other condition involves a notion relevant to a section rigid in the first order in \cite[Definition 3.31]{RZ}.
As a consequence, we use the deformation theory of abelian varieties and $p$-divisible groups.

With the above background in mind, we state our main result.
Let $F$ be a CM field, and $F_0$ be its maximal totally real subfield.
Set $\Phi$ to be a CM type of $F/F_0$.
Let $W$ be a nondegenerate $F/F_0$-hermitian space of positive dimension $n$.
For each element $\varphi$ of $\Phi$, the pair $(r_\varphi,r_{\ol \varphi})$ denotes the signature of the hermitian space $W\otimes_{F,\varphi} \bC$.
We then have the unitary group $G = \oU(W)$ over $F_0$ and the torus over $\bQ$ as follows:
\[
    Z^\bQ(R) = \{z\in\Res_{F/\bQ}\Gm(R)\mid \Nm_{F/F_0}(z)\in R^{\times}\},
\]
$R$ being a $\bQ$-algebra.
Set $\wt G$ to be the direct product of $Z^\bQ$ and $\Res_{F_0/\bQ} G$.
The PEL type Shimura variety $\Sh_{K_{\wt G}}(\wt G,h_{\wt G})$ attached to this group is considered mainly in this paper.
The reasons of adapting this $\wt G$ is in \cite[Remark 2.6]{RSZExp}.

In our case, we have a particular choice of connected parahoric group scheme $\wt\cG$ over $\bZ_p$ with generic fiber $\wt G$ and choose some groups $K_{\wt G}$ whose $p$ part is $\wt\cG(\bZ_p)$ and whose prime-to-$p$ part $K_{\wt G}^p$ is sufficiently small.
Then we have a flat semi-global integral model $\cM_{K_{\wt G}}$ of the above Shimura variety having the level $K_{\wt G}$.
\cite{RSZInt} and \cite{RSZExp} think of only one prime-to-$p$ level in terms of $Z^\bQ$ in their integral models, but here we generalize their auxiliary model to allow various levels so that we can discuss the canonicity of the resulting RSZ models.
We also have a local model $\Mloc$ attached by \cite{PZLoc} to $\wt\cG$ and the minuscule cocharacter of $\wt G$ coming from the Shimura datum.

As we vary $K_{\wt G}^p$, we obtain by construction a system of models $(\cM_{K_{\wt G}})_{K_{\wt G}^p}$ with transition morphisms for $K'^p_{\wt G}\subseteq K_{\wt G}^p$
\[
    \pi_{K'^p_{\wt G},K_{\wt G}^p}\colon \cM_{\wt\cG(\bZ_p)K'^p_{\wt G}}\to \cM_{\wt\cG(\bZ_p)K_{\wt G}^p}
\]
extending the transition morphisms on generic fibers.

\begin{main}[Theorem \ref{main}]
This system is canonical with respect to $(\wt\cG,\Mloc)$.
This means what follows among other things.
\begin{enumerate}
    \item The morphisms $\pi_{K'^p_{\wt G},K_{\wt G}^p}$ are finite etale.
    \item For a discrete valuation ring $R$ of mixed characteristic $(0,p)$, the map
    \[
        \varprojlim_{K_{\wt G}^p}\cM_{K_{\wt G}}(R)\to \varprojlim_{K_{\wt G}^p}\cM_{K_{\wt G}}\left(R\left[\frac{1}{p}\right]\right)
    \]
    is bijective.
    \item For each $K_{\wt G}^p$, there is a locally universal $(\wt\cG, \Mloc)$-display on the $p$-adic formal completion of $\cM_{K_{\wt G}}$ associated with the pro-etale local system on $\Sh_{K_{\wt G}}(\wt G,h_{\wt G})$ given by the covers
    \[
        \Sh_{K'_{\wt G,p}K_{\wt G}^p}(\wt G,h_{\wt G})\to \Sh_{K_{\wt G}}(\wt G,h_{\wt G}),
    \]
    $K'_{\wt G,p}$ ranging over open subgroups of $\wt\cG(\bZ_p)$.
    These $(\wt\cG, \Mloc)$-displays are compatible with respect to the base change by $\pi_{K'^p_{\wt G},K_{\wt G}^p}$.
\end{enumerate}
As a result, the system is isomorphic to the one introduced in \cite[4.2.1]{KPInt}.
\end{main}
We refer to the body of the paper and \cite{PapInt} for unexplained materials some of which are implicit here.

We make a remark about how the nature of RSZ models is reflected in our proof.
The construction of integral models involves delicate conditions, for instance, the refined spin condition \cite[\S 4.4]{RSZInt}, to make the models flat.
However, such conditions does not explictly appear in our proof.
Rather, we use the resulting flatness almost everywhere.
This flatness is necessary to obtain $(\wt\cG, \Mloc)$-displays, for example.

Finally, we turn to the organization of the paper.
In \S \ref{sec:Can}, we review the notion of canonical integral models of Shimura varieties following \cite{PapInt}.
In \S \ref{sec:Shi}, we explain more about the Shimura variety $\Sh_{K_{\wt G}}(\wt G,h_{\wt G})$.
In \S \ref{sec:Aux}, we discuss the auxiliary models that supplement the RSZ integral models.
In \S \ref{sec:Def}, we write down the integral models of Shimura varieties whose canonicity is showed in the rest of the paper.
We address the extension property of the models in \S \ref{sec:Ext} while we construct the locally universal $(\wt\cG, \Mloc)$-displays in \S \ref{sec:LocUniv}.

\subsection*{Acknowledgements}

It is a pleasure to thank my advisor N.~Imai for his academic support resuliting in this paper.
I thank G.~Pappas for answering my questions in emails.
I also thank the referee for suggesting that a direct argument here unlike the previous version of this paper could be possible.

\subsection*{Notation}

Set $p$ to be an odd prime.
If $K$ is a usual or $p$-adic number field, then $\cO_K$ is its ring of integers.
We fix an embedding $\ol\bQ\to \ol{\bQ_p}$.
For a number field $K$, let $\bA_K$, $\bA_{K,f}$ and $\bA_{K,f}^p$ denote the adele of $K$, finite adele of $K$ and finite adele of $K$ away from $p$.
In the three symbols, $K$ is omitted if $K = \bQ$.
For an abelian group $\Lambda$, we put $\wh\Lambda = \displaystyle\varprojlim_{0\neq n\in\bN}\Lambda/n\Lambda$ and $\wh\Lambda^p = \displaystyle\varprojlim_{p\nmid n\in\bN}\Lambda/n\Lambda$.
Phrases like ``over $\Spec A$'' and equations like $X\times_{\Spec A}\Spec B$ are abbreviated as ``over $A$'' and $X\otimes_{A} B$, or simply $X_B$, for example.
If $A$ is an abelian scheme over a scheme $S$, then its $p$-adic Tate module, prime-to-$p$ Tate module and full Tate module are denoted respectively by $T_p(A)$, $T^p(A)$ and $\wh{T}(A)$.
The conventions for $p$-divisible groups are analogous.
Their rational versions are $V_p(A)$, $V^p(A)$ and $\wh{V}(A)$.
The dual abelian scheme of $A$ is written $A^\vee$.
If there is an $\cO_K$-action (resp. an $\cO_{K,(p)}$-action, a $K$-action) $i$ on $A$ for a CM field $K$, the default action of these rings on $A^\vee$ is through $i(\ol a)^\vee$ for each $a$ in these rings.
If $A$ is an abelian scheme or a $p$-divisible group over a scheme $S$, then its Lie algebra is denoted by $\Lie A$.
The notation $\pi_1$ means the etale fundamental groups whereas $\pi_1^{\alg}$ stands for the algebraic fundamental groups of algebraic groups.
The notation $\Sh$ is not for Shimura varieties over $\bC$ but for canonical models.
The moduli schemes and the moduli functors are denoted by the same symbols.
For a finite set $X$, set $\#X$ to be the number of elements of it.

\section{Overview of canonical integral models}\label{sec:Can}

The description below is specific to our case.

\subsection{Notation} \label{sec:Can-Not}

The notation added here is different from that of other subsequent sections.
Let $(G, X)$ be a Shimura datum.
Especially, $G$ is connected.
We assume
\begin{itemize}
    \item that $G$ splits over a tamely ramified extension of $\bQ_p$,
    \item that $p\nmid \#\pi_1^{\alg}(G)$, and
    \item that the center of $G$ has the same $\bQ$-split rank as the $\bR$-split rank.
\end{itemize}
Set $\mu$ to be a minuscule cocharacter of $G_{\ol{\bQ_p}}$ coming from an element of $X$.
Let $E$ be the reflex field of $(G,X)$.
Set $\nu$ to be its $p$-adic place induced by the fixed embedding $\ol\bQ\to\ol{\bQ_p}$.
Let $\fp$ be the maximal ideal of $\cO_{E_\nu}$.
Set $X_\mu$ to be the projective smooth variety over $E$ to which $G_{\ol{\bQ_p}}/P_\mu$ descends.
Here, $P_\mu$ means the parabolic subgroup of $G_{\ol{\bQ_p}}$ corresponding to $\mu$ \cite[Theorem 25.1]{MilAG}.

We take a connected parahoric group scheme $\cG$ associated with a point in the extended Bruhat--Tits building of $G(\bQ_p)$.
In \cite{PZLoc}, Pappas and Zhu construct a local model $\Mloc$ with respect to $\cG$ and the conjugacy class of $\mu$ under our assumption on the splitting of $G$.

We assume that $(\cG, \Mloc)$ is of strongly integral local Hodge type \cite[3.1.4]{PapInt}, i.e., there is a closed immersion $\iota\colon\cG\to\GL_{n,\bZ_p}$ such that
\begin{enumerate}
    \item $\iota(\cG)$ includes the center of $\GL_{n,\bZ_p}$,
    \item $\iota\mu$ is conjugate over $\ol{\bQ_p}$ to the standard minuscule cocharacter sending $a$ to $\diag(a^{(d)}, 1^{(n-d)})$ for some $0\leq d\leq n$, and
    \item the Zariski closure in $\Gr(d,n)_{\cO_{E_\nu}}$ of the image of $\iota_*\colon X_{\mu,E_\nu}\to\Gr(d,n)_{E_\nu}$ induced by $\iota$ is isomorphic to $\Mloc$ including the $\cG$-action.
\end{enumerate}
Such an $\iota$ is called a strongly integral local Hodge embedding.
The notation $\iota_*$ also stands for the embedding $\Mloc\to\Gr(d,n)_{\cO_{E_\nu}}$ in the last condition and its other variants.

\subsection{Dieudonn\'e $(\cG, \Mloc)$-displays} \label{sec:Can-Dis}

Let $R$ be a complete noetherian local flat $\cO_{E_{\nu}}$-algebra with perfect residue field of characteristic $p$.
Recall from \cite[\S 1, \S 2]{ZinCFT}, etc., that in this case we have a variant $\hat{W}(R)$ of the ring of Witt vectors equipped with a Frobenius $\varphi$\footnote{Not to be confused with elements of the CM type.}.
We also recall the functor of \cite[Proposition 4.1.2]{PapInt}.
It is from the groupoid of pairs $(\cP,q)$ of a $\cG$-torsor $\cP$ over $\hat W(R)$ and a $\cG$-equivariant morphism $q\colon\cP\otimes_{\hat W(R)}R\to \Mloc$ of $\cO_{E_\nu}$-schemes.
It lands in the groupoid of tuples $(\cQ', \cQ,\alpha)$ of two $\cG$-torsors $\cQ'$, $\cQ$ over $\hat W(R)$ and a $\cG$-equivariant morphism $\alpha\colon\cQ[1/p]\to\varphi^*\cQ'[1/p]$.
In addition, the functor sends a pair $(\cP, q)$ to a tuple of the form $(\cP,\cQ, \alpha)$.

\begin{defi}[{\cite[Definition 4.3.2]{PapInt}}]
A \emph{Dieudonn\'e $(\cG, \Mloc)$-display} over $R$ is a tuple of
\begin{itemize}
    \item a $\cG$-torsor $\cP$ over $\hat W(R)$,
    \item a $\cG$-equivariant morphism $q\colon\cP\otimes_{\hat W(R)}R\to \Mloc$ and
    \item a $\cG$-isomorphism $\Psi\colon\cQ\isoarrow\cP$, where $\cQ$ is the second constituent of the image of $(\cP,q)$ by the above functor.
\end{itemize}
\end{defi}

We turn to $(\cG, \Mloc)$-displays.
Let $R_0$ be a $p$-adically complete flat $\cO_{E_\nu}$-algebra formally of finite type.
The above functor sending $(\cP, q)$ to $(\cP, \cQ, \alpha)$ is analogously defined when $\hat W(R)$ is replaced by $W(R_0)$.
\begin{defi}[{\cite[Definition 4.2.2]{PapInt}}]
A \emph{$(\cG, \Mloc)$-display} over $R_0$ is a tuple of
\begin{itemize}
    \item a $\cG$-torsor $\cP$ over $W(R_0)$,
    \item a $\cG$-equivariant morphism $q\colon\cP\otimes_{W(R_0)}R_0\to \Mloc$ and
    \item a $\cG$-isomorphism $\Psi\colon\cQ\isoarrow\cP$, where $\cQ$ is the second constituent of the image of $(\cP,q)$ by the above functor.
\end{itemize}
\end{defi}

This definition is generalized in \cite[4.2.4]{PapInt} for a $p$-adic flat formal scheme $\fX$ over $\Spf\cO_{E_\nu}$ formally of finite type by substituting a sheaf of rings $W(\cO_\fX)$ for $W(R_0)$.

\subsection{Rigidity and locally universal displays} \label{sec:Can-Rig}

Assume that the local ring $(R,\fm, k)$ in \S \ref{sec:Can-Dis} is normal and formally of finite type over $W(k_0)$ for some perfect field $k_0$ of characteristic $p$.
Set $(\cP, q,\Psi)$ to be a Dieudonn\'e $(\cG, \Mloc)$-display over $R$.

We refer to \cite[Definition 4.5.8]{PapInt} for the definition of a section $s\in\cP(\hat{W}(R))$ being rigid in the first order at $\fm$.
\begin{defi}[{\cite[Definition 4.5.10]{PapInt}}]
The Dieudonn\'e $(\cG, \Mloc)$-display $(\cP, q,\Psi)$ is called \emph{locally universal} if there exists a section $s\in \cP(\hat W(R))$ rigid in the first order such that $q\circ (s\otimes 1)\colon \Spec R\to \cP\otimes_{\hat W(R)}R\to \Mloc\otimes_{\bZ_p}W(k)$ induces an isomorphism from the completion of the local ring at the image of $\fm$ to $R$.
\end{defi}

In \cite[\S 4.4]{PapInt}, a classical Dieudonn\'e display $\cD(\iota)$ is attached to the tuple $\cD = (\cP, q,\Psi)$, which in turn gives a $p$-divisible group $\sG$ by the main theorem of \cite{ZinCFT}.

\subsection{Canonical integral models}
Let $K^p$ be a sufficiently small compact open subgroup of $G(\bA_f^p)$.
Put $K_p = \cG(\bZ_p)$ and $K = K_pK^p\subseteq G(\bA_f)$.

Let $\sS_K$ be a normal flat separated $\cO_{E,(\nu)}$-schemes of finite type with generic fiber $\Sh_K(G,X)$.
Let $\cL_K$ be the pro-etale $\cG(\bZ_p)$-local system on $\Sh_K(G,X)$ that is given by the covers $(\Sh_{K'_pK^p}(G,X)\to\Sh_{\cG(\bZ_p)K^p}(G,X))_{K'_p\subseteq\cG(\bZ_p)}$.
Suppose that we are given a Dieudonn\'e $(\cG, \Mloc)$-display $\sD_x$ over the strict completion $R_x$ of the local ring of $\sS_K$ at each $x\in\sS_K(\ol{\bF_p})$.
As in \S \ref{sec:Can-Rig}, we associate a $p$-divisible group $\sG_x$ over $R_x$ with $\sD_x$.

We cite \cite[Definition 6.1.5]{PapInt} for the notion that $\cL_K$ and $\sD_x$ are associated, and for the definition of $(\cL_K, (\sD_x, \alpha_x)_{x\in\sS_K(\ol{\bF_p}))})$ being an associated system.
Here, $\alpha_x$ is an isomorphism
\[
    T_p\left(\sG_x\otimes_{R_x} R_x\left[\frac{1}{p}\right]\right)\isoarrow\cL_{K,R_x\left[\frac{1}{p}\right]}\times^{\cG}\bA_{\bZ_p}^n
\]
of pro-etale $\bZ_p$-local systems on $R_x[1/p]$, where the push-out on the right hand side is via $\iota$.
These definitions do not depend on $\iota$ by \cite[Proposition 6.5.1]{PapInt}.

\begin{thm}[{\cite[Theorem 6.4.1]{PapInt}}]\label{exist}
Suppose that there exists a $p$-divisible group on $\sS_K$ whose pull-back over $\Sh_K(G,X)$ has the $p$-adic Tate module isomorphic to $\cL_K\times^{\cG}\bA_{\bZ_p}^n$.
Then, there exists a unique associated system $(\cL_K, (\sD_x, \alpha_x)_{x\in\sS_K(\ol{\bF_p}))})$ up to unique isomorphism.
\end{thm}

We also consider a situation where we have one $(\cG, \Mloc)$-display instead of pointwise Dieudonn\'e $(\cG, \Mloc)$-displays $\sD_x$.
\begin{defi}[{\cite[Definition 6.2.1, Definition 6.2.2]{PapInt}}]
Assume that $\sD_K$ is a $(\cG, \Mloc)$-display over the $p$-adic formal completion $\fS_K$ of $\sS_K$.
Then $\sD_K$ is \emph{associated with $\cL_K$} and \emph{locally universal} over $\fS_K$ if for each $x\in\sS_K(\ol{\bF_p})$, there is a locally universal Dieudonn\'e $(\cG, \Mloc)$-display $\sD_x$ over $R_x$ associted with $\cL_K$ such that the $(\cG, \Mloc)$-display $\sD_x\otimes_{\hat W(R_x)} W(R_x)$ is isomorphic to $\sD_K\otimes_{W(\cO_{\fS_K})} W(R_x)$.
\end{defi}

We slightly modify \cite[Definition 7.1.3]{PapInt} in response to our moduli problems.

\begin{defi} \label{candef}
Let $(\sS_K)_{K^p}$ be a system of $\cO_{E,(\nu)}$-schemes as above for $K^p$ running over a cofinal subset of the set of compact open subgroups of $G(\bA_f^p)$, together with morphisms
\[
    \pi_{K'^p,K^p}\colon\sS_{K_pK'^p}\to \sS_{K_pK^p}
\]
for $K'^p\subseteq K^p$ extending $\Sh_{K_pK'^p}(G,X)\to \Sh_{K_pK^p}(G,X)$.
Then the system is called \emph{canonical} if
\begin{enumerate}
    \item for each $K'^p\subseteq K^p$, the morphism $\pi_{K'^p,K^p}$ is finite etale,
    \item\label{extension} for any discrete valuation ring $R$ of mixed characteristic $(0,p)$, the map
    \[
        \varprojlim_{K^p}\sS_K(R)\to \varprojlim_{K^p}\sS_K\left(R\left[\frac{1}{p}\right]\right)
    \]
    is bijective, and
    \item\label{strongdisplay} for each $K^p$, there is a locally universal $(\cG, \Mloc)$-display $\sD_K$ on the $p$-adic completion $\fS_K$ associated with $\cL_K$.
    These $(\cG, \Mloc)$-displays are compatible with respect to the base change by $\pi_{K'^p,K^p}$.
\end{enumerate}
\end{defi}

As an example, if $(G,X)$ is of Hodge type, then the integral models of $\Sh_K(G,X)$ by Kisin and Papas \cite[4.2.1]{KPInt} are canonical \cite[Theorem 8.2.1]{PapInt}.

\section{The Shimura variety}\label{sec:Shi}

We review some of \cite[\S 2, \S 3]{RSZExp}.

\subsection{Notation}
The convention here is the same as in \S \ref{sec:Intro}, and valid through the rest of the paper.
The same is true for the symbols in the upcoming sections.
Let $F$ be a CM field with $F_0$ its maximal totally real subfield.
Set $\Phi$ to be a CM type of $F/F_0$.
We fix a totally negative element $\Delta$ of $F_0$ and its square root $\sqrt{\Delta}$ such that an embedding $\varphi$ of $F$ in $\bC$ is in $\Phi$ if and only if $\varphi(\sqrt{\Delta})\sqrt{-1}^{-1} >0$.

Let $W$ be a nondegenerate $F/F_0$-hermitian space of positive dimension $n$.
For each $\varphi\in\Phi$, the pair $(r_\varphi,r_{\ol \varphi})$ denotes the signature of the hermitian space $W\otimes_{F,\varphi} \bC$.
Set $G$ to be the unitary group $\oU(W)$ over $F_0$.
We also define a group scheme $Z^\bQ$ over $\bZ$, $G^\bQ$ over $\bQ$ and $\wt G$ over $\bQ$ as
\begin{align*}
    Z^\bQ(R) &\colonequals \{z\in\Res_{\cO_F/\bZ}\Gm(R)\mid \Nm_{\cO_F/\cO_{F_0}}(z)\in R^{\times}\},\\
    G^\bQ(R) &\colonequals \{g\in\Res_{F_0/\bQ}\GU (W)(R)\mid c(g)\in R^{\times}\},\\
	\wt G &\colonequals Z^\bQ\times_{\Gm}G^\bQ,
\end{align*}
where $R$ is a ring (resp. a $\bQ$-algebra) and $c$ means the similitude.

The groups $\wt G$ and $Z^\bQ\times\Res_{F_0/\bQ} G$ are isomorphic by sending $(z,g)$ to $(z,z^{-1}g)$.
We also note that $\wt G$ satisfies the last two among the three assumptions imposed on $G$ in \S \ref{sec:Can-Not}.
We assume that the first first assumption also holds for $\wt G$.
This is true in the case of \cite[\S 4.1, \S 4.4]{RSZInt}.

Let $\bS = \Res _{\bC/\bR} \Gm$ be the Deligne torus.
If $h\colon\bS\to H$ is a morphism of algebraic groups over $\bR$, then $\{h\}$ denotes its conjugacy class.

\subsection{The Shimura data}

For $Z^\bQ$, we identify $\bR\otimes_{\bQ} F$ with $\displaystyle\prod_{\varphi\in\Phi} \bC$ through elements of $\Phi$, getting an isomorphism between $(Z^\bQ)_\bR$ and a subgroup of $\displaystyle\prod_{\varphi\in\Phi}\bS$.
We let $h_{Z^\bQ}\colon\bS\to (Z^\bQ)_\bR$ send $z$ in $\bC^{\times}$ to $(z)_{\varphi\in\Phi}$.
The pair $(Z^\bQ, \{h_{Z^\bQ}\})$ is a Shimura datum.
Let $E_\Phi$ be its reflex field, or the reflex field of $\Phi$.

We next turn to $G^\bQ$.
Identify $W\otimes_{F,\varphi} \bC$ with $\bC^n$ so that the hermitian form on $W\otimes_{F,\varphi} \bC$ is represented by $\diag(1^{(r_\varphi)},-1^{(r_{\ol\varphi})})$.
Then, we have $h_{G^\bQ}\colon \bS\to (G^\bQ)_\bR$, where $h_{G^\bQ}(\bR)$ is induced by an $\bR$-algebra homomorphism $\bC\to\displaystyle\prod_{\varphi\in\Phi} \End(\bC^n)$ sending $\sqrt{-1}$ to $(\sqrt{-1}\diag(1^{(r_\varphi)},-1^{(r_{\ol\varphi})}))_{\varphi}$.
The pair $(G^\bQ,\{h_{G^\bQ}\})$ is a Shimura datum.

We finally come to $\wt G$.
Let $h_{\wt G}$ denote the map $h_{Z^\bQ}\times_{\Gm} h_{G^\bQ}\colon\bS\to \wt G_\bR$.
The pair $(\wt G,\{h_{\wt G}\})$ is a Shimura datum.
Set $E\subseteq \bC$ to be its reflex field.
The Shimura datum is of PEL type, but we just describe its Hodge embedding and its integral version in \S \ref{sec:Def-Emb}.

\section{Auxiliary models}\label{sec:Aux}
Set $\nu$ to be the $p$-adic place of $E$ induced by the fixed embedding $\ol\bQ\to \ol{\bQ_p}$.
Let $\fa\neq 0$ be an ideal of $\cO_{F_0}$ prime to $p$ such that the following category $\cM_0^{\fa}$ fibered in groupoids over the category of locally noetherian $\cO_E$-schemes is nontrivial.
Such $\fa$ exists by \cite[Remark 3.7(iii)]{RSZExp}.

Viewed as a pseudo-functor, $\cM_0^\fa$ associates with an object $S$ of the source category the groupoid of tuples $(A_0,i_0,\lambda_0)$ of
\begin{itemize}
    \item an abelian scheme $A_0$ over $S$,
    \item $i_0\colon \cO_F\to \End_S(A_0)$ with the Kottwitz condition that for every $a\in\cO_F$, the characteristic polynomial of $i_0(a)$ acting on $\Lie A_0$ is $\displaystyle\prod_{\varphi\in\Phi}(T-\varphi(a))$, and
    \item $\lambda_0\colon A_0\to A_0^\vee$ is an $\cO_F$-linear polarization with kernel $A_0[\fa]$.
\end{itemize}
Arrows of this groupoid are isomorphisms of abelian schemes preserving the other structures.
This moduli problem is represented by a finite etale $\cO_E$-scheme by applying the argument of \cite[Proposition 3.1.2]{HowCyc}.

We briefly recall $\cM_0^{\fa,\xi}$ in \cite[(4.2)]{RSZExp}.
Via the first homology group, $\cM_0^\fa(\bC)$ bijects to the set $\cL_\Phi^\fa$ of the isomorphism classes of the following pairs $(\Xi_0,\langle *,*\rangle_0)$.
This pair comprises an invertible $\cO_F$-module $\Xi_0$ and an alternating $\cO_F/\cO_{F_0}$-balanced form $\langle *,*\rangle_0\colon\Xi_0\times\Xi_0\to \bZ$ such that $\langle\sqrt{\Delta}*,*\rangle_0$ is positive definite and the dual of $\Xi_0$ with respect to $\langle *,*\rangle_{0,\bQ}$ is $\fa^{-1}\Xi_0$.
Elements $\Xi_0,\Xi'_0\in\cL_\Phi^\fa$ are called equivalent if there are $\cO_F$-linear isomorphisms $\Xi_0\otimes_\bZ \wh\bZ\isoarrow\Xi'_0\otimes_\bZ \wh\bZ$ and $\Xi_0\otimes_\bZ \bQ\isoarrow\Xi'_0\otimes_\bZ\bQ$ by which alternating forms correspond up to a constant in $\wh\bZ^{\times}$ (resp. $\bQ^{\times}$).
The equivalence relation is denoted by $\sim$.
If $\xi$ is in $\cL_\Phi^\fa/\mathord{\sim}$, then $\cM_0^{\fa,\xi}$ is the Zariski closure in $\cM_0^\fa$ of the set of the image of $x$ in $\cM_0^\fa$ where $x\in \cM_0^\fa(\bC)$ corresponds to an element of $\xi$.
This $\cM_0^{\fa,\xi}$ is clopen in $\cM_0^\fa$.

Let $(\Lambda_{0},\psi_0)$ be an element of $\cL_\Phi^\fa$.
Put $K_{Z^\bQ} = Z^\bQ(\bZ_p)K_{Z^\bQ}^p$ for an open subgroup $K_{Z^\bQ}^p$ of $Z^\bQ(\wh\bZ^p)$.
If $S$ is a locally noetherian $\cO_{E,(p)}$-scheme and $(A_0,i_0,\lambda_0)$ is an object of $\cM_0^\fa(S)$, then let $\ul{\Isom}((\wh{\Lambda_0}^p)_S,T^p(A_0))$ be the pro-etale sheaf of $\wh{\cO_F}^p$-linear isomorphisms from the constant sheaf $(\wh{\Lambda_0}^p)_S$ over $S$ to the pro-etale sheaf $T^p(A_0)$ transferring $\psi_0$ to the Riemann form of $\lambda_0$ up to a constant in $\wh \bZ^{p,\times}$.

We define the following category $\cM_{0,K_{Z^\bQ}}^{\Lambda_0}$ fibered in groupoids over the category of locally noetherian $\cO_{E,(p)}$-schemes.
As a pseudo-functor, a locally noetherian $\cO_{E,(p)}$-scheme $S$ is sent to the groupoid of tuples $(A_0,i_0, \lambda_0, \epsilon^p)$ composed of
\begin{itemize}
    \item an object $(A_0,i_0, \lambda_0)$ of $\cM_0^\fa(S)$ and
    \item a section $\epsilon^p$ over $S$ of the pro-etale sheaf
    \[
        \ul{\Isom}((\wh{\Lambda_0}^p)_S,T^p(A_0))/K_{Z^\bQ}^p.
    \]
\end{itemize}
Arrows of this groupoid are the ones of $\cM_0^\fa$ by which the level structures coincide.
Note that $K_{Z^\bQ}^p$ acts on $\ul{\Isom}((\wh{\Lambda_0}^p)_S,T^p(A_0))$ by the action on $\wh{\Lambda_0}^p$.
Also, if $S$ is connected with a geometric point $s\to S$, then $\epsilon^p$ can be identified with a $K_{Z^\bQ}^p$-class of $\wh{\cO_F}^p$-linear isomorphisms $\wh{\Lambda_0}^p\isoarrow T^p(A_{0,s})$ transferring $\psi_0$ to the Riemann form of $\lambda_0$ up to a constant in $\wh\bZ^{p,\times}$ and the $\pi_1(S,s)$-action on the target to actions on the source by elements of $K_{Z^\bQ}^p$.

Toward the representability of the auxiliary models, we first observe in the sequel that $\cM_{0,Z^\bQ(\wh\bZ)}^{\Lambda_0}$ gives a fibered full subcategory of $\cM_0^\fa\otimes_\bZ \bZ_{(p)}$.
Indeed, the level structure with respect to $Z^{\bQ}(\wh\bZ)$ attached to $(A_0,i_0,\lambda_0)$ in $\cM_0^\fa(S)$ is unique if exists since the group of automorphisms of $\wh{\Lambda_0}^p$ preserving the $\wh{\cO_F}^p$-action and $\psi_0$ up to a scalar is $Z^\bQ(\wh\bZ^p)$.
We identify $\cM_{0,Z^\bQ(\wh\bZ)}^{\Lambda_0}$ with its essential image in $\cM_0^\fa\otimes_\bZ \bZ_{(p)}$.

We also relate $\cM_{0,Z^{\bQ}(\wh\bZ)}^{\Lambda_0}$ to $\cM_0^{\fa,\xi}$.
Set $\xi$ to be the element $[(\Lambda_0,\psi_0)]$ of $\cL_\Phi^\fa/\mathord{\sim}$.
Then we claim that $\cM_{0,Z^{\bQ}(\wh\bZ)}^{\Lambda_0}$ contains $\cM_0^{\fa,\xi}\otimes_{\bZ}\bZ_{(p)}$.
In fact, we have only to construct a level structure on $(A_0,i_0,\lambda_0)\in\cM_0^{\fa,\xi}(S)$ for each connected locally noetherian $\cO_{E,(p)}$-scheme $S$.
If $S$ has a point with residue field of characteristic $0$, then $S$ has a geometric point $s$ such that the corresponding element of $\cM_0^{\fa,\xi}(k(s))$ factors through a $\bC$-valued point $t$ of $\cM_0^{\fa,\xi}$.
We can think of a complex torus $A_0(\bC)$ via $t$.
Its first homology group over $\wh\bZ$ is isometric to $\wh{\Lambda_0}$ up to a similitude in $\wh\bZ^{\times}$, and that translates into the level structure $\wh{\Lambda_0}^p\isoarrow T^p(A_{0,s})$.
We can deal with the other case by taking a lift to characteristic $0$, relying on the finite etaleness of $\cM_0^{\fa,\xi}$.
We obtain a decomposition
\[
    \cM_{0,Z^\bQ(\wh\bZ)}^{\Lambda_0} = \bigsqcup\cM_{0}^{\fa,[\Xi_0]},
\]
where the disjoint sum ranges over $[\Xi_0]\in\cL_\Phi^\fa/\mathord{\sim}$ such that $\wh{\Xi_0}^p$ is isometric to $\wh{\Lambda_0}^p$ up to a similitude in $(\wh{\bZ}^p)^{\times}$.
In particular, the left hand side is a finite etale scheme over $\Spec_{\cO_{E,(p)}}$.

\begin{prop} \label{auxiliary finet}
$\cM_{0,K_{Z^\bQ}}^{\Lambda_0}$ is represented by a scheme finite etale over $\cM_{0,Z^\bQ(\wh\bZ)}^{\Lambda_0}$.
\end{prop}

\begin{proof}
Let $S$ be a connected locally noetherian $\cO_{E,(p)}$ with an object $(A_0,i_0,\lambda_0,\epsilon^p)$ of $\cM_{0,Z^\bQ(\wh\bZ)}^{\Lambda_0}(S)$.
We fix a representative $\wh{\Lambda_0}^p\to T^p(A_{0,s})$ of $\epsilon^p$ denoted by the same symbol.
Set $\cN$ to be the fiber product of $\cM_{0,K_{Z^\bQ}}^{\Lambda_0}$ and $S$ over $\cM_{0,Z^\bQ(\wh\bZ)}^{\Lambda_0}$.
Then, $\cN$ can be seen as a set-valued functor from the category of locally noetherian $S$-schemes, sending an $S$-scheme $T$ to the inverse image of $\epsilon^p_T$ by the map
\[
    (\ul{\Isom}((\wh{\Lambda_0}^p)_S,T^p(A_0))/K_{Z^\bQ}^p)(T)\to (\ul{\Isom}((\wh{\Lambda_0}^p)_S,T^p(A_0))/Z^\bQ(\wh\bZ^p)))(T).
\]

For a given morphism $f\colon T\to S$ of connected locally noetherian schemes and a geometric point $t$ of $T$, we would like to find out the set of $\epsilon'\in Z^{\bQ}(\wh\bZ^p)/K_{Z^\bQ}^p$ with the condition below.
Let $g\colon\pi_1(S,t)\to Z^\bQ(\wh\bZ^p)$ be induced by $\epsilon^p$.
The condition is that $f_*\colon\pi_1(T,t)\to \pi_1(S,t)$ composed with $\epsilon'^{-1}g\epsilon'$ factors through $K_{Z^{\bQ}}^p$.
Although $\epsilon'^{-1}g\epsilon' =g$ holds since $Z^\bQ$ is commutative, writing $\epsilon'$ explicitly clarifies the proof.

Set $G_{\epsilon'}$ to be the inverse image of $K_{Z^\bQ}^p$ by $\epsilon'^{-1}g\epsilon'$.
Then $\pi_1(S,t)/G_{\epsilon'}$ is a $\pi_1(S,t)$-set in correspondence with a finite etale cover $S'\to S$.
Lifts of $f$ to $T\to S'$ are in bijection with $zG_{\epsilon'}\in\pi_1(S,t)/G_{\epsilon'}$ fixed by $f_*(\pi_1(T,t))$.
For $z\in\pi_1(S,t)$, the stabilizer of $zG_{\epsilon'}\in\pi_1(S,t)/G_{\epsilon'}$ in $\pi_1(S,t)$ is $zG_{\epsilon'}z^{-1} = G_{g(z)\epsilon'}$.
Thus, $\cN$ is represented by the disjoint sum of $S'$ over $\epsilon'\in g(\pi_1(S,t))\backslash Z^\bQ(\wh\bZ^p)/K_{Z^\bQ}^p$.
\end{proof}

Finally, the following description \cite[\S 5]{KotInt} of $\Sh_{K_{Z^\bQ}}(Z^\bQ, \{h_{Z^\bQ}\})$ and the comparison of $\bC$-valued points \cite[Theorem 8.17]{MilShi} show that the generic fiber of
\[
    \cM_0 = \cM_{0,K_{Z^\bQ}}^{\Lambda_0}\times_{\cM_0^\fa} \cM_{0}^{\fa,[\Lambda_0]}
\]
is $\Sh_{K_{Z^\bQ}}(Z^\bQ, \{h_{Z^\bQ}\})\otimes_{E_\Phi}E$.
The scheme $\Sh_{K_{Z^\bQ}}(Z^\bQ, \{h_{Z^\bQ}\})$ is a clopen subscheme of the $E_\Phi$-scheme representing a functor from the category of locally noetherian $E_\Phi$-schemes to that of groupoids.
The functor sends such a scheme $S$ to the groupoid of tuples $(A_0,i_0,\lambda_0,\epsilon)$, where
\begin{itemize}
    \item $A_0$ is an abelian scheme over $S$,
    \item $i_0\colon F\to\End_S(A_0)\otimes_\bZ \bQ$ satisfies the Kottwitz condition with respect to $\Phi$ as in \cite[\S 5]{KotInt},
    \item $\lambda_0\in\Hom(A_0,A_0^\vee)\otimes_\bZ \bQ$ is an $F$-linear $\bQ$-polarization, and
    \item $\epsilon\in(\ul{\Isom}(\Lambda_0\otimes_\bZ \bA_f,\wh V(A_0))/K_{Z^\bQ})(S)$,
\end{itemize}
with arrows comprising $F$-linear $\bQ$-isogenies of abelian schemes preserving the level structures and sending polarizations with each other up to a factor in $\bQ_S^\times(S)$.
This description of $\Sh_{K_{Z^\bQ}}(Z^\bQ, \{h_{Z^\bQ}\})$ works even when the level at $p$ is not $Z^{\bQ}(\bZ_p)$.

\section{RSZ integral models}\label{sec:Def}

Our formulation is based on \cite[\S 4.1, \S 5.2]{RSZExp}.

\subsection{Notation}

For each $p$-adic place $w$ of $F$, we fix a uniformizer $\pi_w$ of $F_w$.
We take them so that if a place $v$ of $F_0$ splits into $w$ and $\ol w$ in $F$, then $\pi_w = \pi_{\ol w}$ as elements of $F_{0,v}$.
For a $p$-adic place $v$ of $F_0$, we let $\pi_v$ be $\pi_w$ if $w$ is the unique place of $F$ above $v$, and an element of $F_v$ sent to $(\pi_w, \pi_{\ol w})$ in $F_w\times F_{\ol w}$ if $v$ splits into $w$ and $\ol w$ in $F$.

Let $(*,*)_0\colon\Lambda_0\times\Lambda_0\to\cO_F$ be the hermitian form such that $\psi_0 = \Tr_{\cO_F/\bZ}\sqrt{\Delta}^{-1} (*,*)_0$.
Likewise, we view $W$ also as a symplectic space via $\psi = \Tr_{F/\bQ}\sqrt{\Delta}^{-1}(*,*)$, the form $(*,*)$ being the hermitian one of $W$.
We take an $\cO_F$-lattice $\Lambda$ of $W$ such that for each $p$-adic place $v$ of $F_0$, the inclusion $\Lambda_v\subseteq \Lambda_v^\vee\subseteq \pi_v^{-1}\Lambda_v$ holds, the dual being the symplectic one.
Regarding the levels, set $K_G^p$ to be a sufficiently small open compact subgroup of $G(\bA_f^p)$.
We introduce the following groups:
\begin{align*}
    K_{G,v} &\colonequals \Stab_{G(F_{0,v})}(\Lambda_v), \\
    K_{G,p} &\colonequals \prod_{v\mid p}K_{G,v}, \\
    K_G &\colonequals K_{G,p}K_G^p, \\
    K_{\wt G} &\colonequals K_{Z^\bQ}K_G\subseteq \wt G(\bA_f).
\end{align*}
We set $\wt\cG$ to be the subgroup $K_{Z^\bQ,p}K_{G,p}$ of $\wt G(\bQ_p)$, seen as a connected parahoric group scheme associated with a point in the extended Bruhat--Tits building of $\wt G(\bQ_p)$.
Finally, we define $L = \Hom_{\cO_F}(\Lambda_0,\Lambda)$.
The $F$-vector space $V = L_{\bQ}$ admits a hermitian form $V\times V\to F$ such that for each $m,m'$ in $\Lambda_0$ and $x,x'$ in $L$,
\[
    (x,x')(m,m')_0 = (xm,x'm'),
\]
where all the forms are hermitian.

\subsection{A Hodge embedding and its integral variant} \label{sec:Def-Emb}

We introduce an algebraic group $\GSp$ over $\bZ_{(p)}$ by the following:
\[
    \GSp(R)\colonequals \{g\in  \GL_R(\Lambda_{0,R}\oplus \Lambda_R)\mid \psi_0(g*,g*) + \psi(g*,g*)\in R^\times (\psi_0 + \psi)\}
\]
for a $\bZ_{(p)}$-algebra $R$.
Also, $S^\pm$ denotes the corresponding Siegel double space.

The group $\wt G$ is embedded into $\GSp_\bQ$ since each component of an element of $\wt G$ preserves $\psi_0$ or $\psi$.
This gives a Hodge embedding $(\wt G,\{h_{\wt G}\})\to (\GSp, S^\pm)$ since $(\psi_0 + \psi)(*, h_{\wt G}(*))$ is negative definite.

We next discuss the strongly integral local Hodge embedding.
We think of $\GL(\Lambda_{0, \bZ_p}\oplus\Lambda_{\bZ_p})$ as an algebraic group over $\bZ_p$ identified with $\GL_{[F:\bQ](n+1),\bZ_p}$.
Set $\iota\colon\wt\cG\to\GL(\Lambda_{0, \bZ_p}\oplus\Lambda_{\bZ_p})$ to be the inclusion.
Over $\bQ_p$, this $\iota$ is the same as the Hodge embedding composed with the inclusion $\GSp_{\bQ_p}\to \GL (\Lambda_{0,\bQ_p}\oplus \Lambda_{\bQ_p})$.
The pair $(\wt\cG,\Mloc)$ is of strongly integral local Hodge type by $\iota$, as we now explain.
The first two conditions on $\iota$ in \S \ref{sec:Can-Not} are straightforward.
The last one is by \cite[Proposition 2.3.7]{KPInt}.

\subsection{Naive models}\label{sec:Def-Nai}

For two abelian schemes $A_0$ and $A$ over a scheme $S$ with $\cO_{F,(p)}$-actions up to prime-to-$p$ isogeny, we define $T^p(A_0,A)$ and $V^p(A_0,A)$ to be the pro-etale sheaf $\Hom_{\wh{\cO_F}^p}(T^p(A_0), T^p(A))$ and $\Hom_{\bA_{F,f}^p}(V^p(A_0), V^p(A))$, respectively.
If $A_0$ and $A$ are endowed with $\cO_{F,(p)}$-linear $\bQ$-polarizations $\lambda_0$ and $\lambda$ respectively, then we have a hermitian form on the sheaf $V^p(A_0,A)$ sending $x,x'\in V^p(A_0,A)$ to $\lambda_0^{-1}\circ x'^\vee\circ\lambda\circ x\in\End_S(V^p(A_0))\isoarrow (\bA_{F,f}^p)_S(S)$.

We define the category $\cM_{K_{\wt G}}^{\mathrm{naive}}$ fibered in groupoids over the category of locally noetherian $\cO_{E,(p)}$-schemes.
It associates with each such scheme $S$ the groupoid of tuples $(A_0,i_0,\lambda_0,\epsilon^p,A,i,\lambda,\eta^p)$, where
\begin{itemize}
    \item $(A_0,i_0,\lambda_0,\epsilon^p)\in\cM_0(S)$,
    \item $A$ is an abelian scheme over $S$,
    \item $i\colon\cO_F\to\End_S(A)$ satisfies another Kottwitz condition that for every $a\in\cO_F$, the characteristic polynomial of $i(a)$ acting on $\Lie A$ is $\displaystyle\prod_{\varphi\in\Hom(F,\ol\bQ)}(T-\varphi(a))^{r_\varphi}$,
    \item $\lambda\in\Hom_{\cO_F}(A,A^\vee)\otimes_{\bZ}\bZ_{(p)}$ is a $\bQ$-polarization, and
    \item $\eta^p$ is a section over $S$ of the pro-etale sheaf \[
    \ul{\HermIsom}((\wh L^p)_S,T^p(A_0,A))/K_G^p.
    \]
    Here, $\ul{\HermIsom}((\wh L^p)_S,T^p(A_0,A))$ is the pro-etale sheaf of $\wh{\cO_F}^p$-linear isomorphisms between the two sheaves preserving the hermitian forms on $L\otimes_{\bZ}\bA_f^p$ and $V^p(A_0,A)$.
\end{itemize}
The tuple is further required to satisfy the two conditions below.
\begin{itemize}
    \item We have a decomposition with respect to $p$-adic places of $F_0$
    \[
        A[p^\infty] = \prod_{v\mid p}A[v^\infty],
    \]
    where the quotient $\cO_{F_{0,v}}$ of $\cO_{F_0}\otimes_{\bZ}\bZ_p$ acts on $A[v^\infty]$ through $i$.
    For each $p$-adic place $v$ of $F_0$, the kernel of $\lambda_v\colon A[v^\infty]\to A^{\vee}[v^\infty]$ induced by $\lambda$ should have the rank $\#(\Lambda_v^\vee/\Lambda_v)$ and should be included in $A[\pi_v] = \Ker (i(\pi_v)\colon A\to A)$.
    \item For every $p$-adic place $v$ of $F_0$ nonsplit in $F$ and each geometric point $s\to S$, \cite[Appendix A]{RSZInt} associates the sign invariant
    \[ \inv_v^r(A_{0,s},i_{0,s},\lambda_{0,s},A_s,i_s,\lambda_s)
    \]
    in $\{\pm 1\}$. This should equal
    \[
        (-1)^{n(n-1)/2}\det V_v
    \]
    in $F_{0,v}^\times/\Nm F_v^\times\simeq\{\pm 1\}$.
    Here, the determinant is the one of a matrix representation of the hermitian form of $V_v$.
\end{itemize}
The latter condition is called the sign condition.
Arrows of this groupoid are pairs consisting of an arrow in $\cM_0$ and an isomorphism of the other abelian schemes preserving the other data.
We note that $\eta^p$ admits a similar identification to $\epsilon^p$ described in terms of stalks.

We compare this moduli with its analogue described by isogenies.
The naive model in \cite{RSZExp} means this isogeny version.
Let $\cR_{K_{\wt G}}^{\mathrm{naive}}$ be the category fibered in groupoids over the category of locally noetherian $\cO_{E,(p)}$-schemes which sends such a scheme $S$ to the groupoid of tuples $(A_0,i_0,\lambda_0,\epsilon^p,A,i,\lambda,\eta^p)$ of
\begin{itemize}
    \item $A_0$, $i_0$, $\lambda_0$, $\epsilon^p$ and $A$ as before,
    \item $i\colon\cO_{F,(p)}\to\End_S(A)\otimes_{\bZ}\bZ_{(p)}$ with the Kottwitz condition,
    \item a $\bQ$-polarization $\lambda\in\Hom_{\cO_F}(A,A^\vee)\otimes_{\bZ}\bZ_{(p)}$, and
    \item a section $\eta^p$ over $S$ of the pro-etale sheaf \[
    \ul{\HermIsom}((L\otimes_{\bZ}\bA_f^p)_S,V^p(A_0,A))/K_G^p
    \]
    with essentially the same definition of $\ul{\HermIsom}$ as its counterpart.
\end{itemize}
The tuple is requested to satisfy the two conditions as before.
Here, the sign invariant at a geometric point $s\to S$ is also defined when the involved actions are up to isogeny.
Moreover, the invariant only depends on the isogeny class of $(A_{0,s},i_{0,s},\lambda_{0,s},A_s,i_s,\lambda_s)$, and is locally constant with respect to $s$.

Arrows of this groupoid are composed of pairs of an arrow in $\cM_0(S)$ and an $\cO_{F,(p)}$-linear $\bZ_{(p)}^\times$-isogeny between the other abelian schemes which preserves the other data.

\begin{lem}[\cf{\cite[Corollary 1.3.5.4]{LanCpt}}]\label{lem:CptSub}
Set $S$ to be a connected locally noetherian scheme of residual characteristic $0$ or $p$ with a geometric point $s\to S$.
Set $A$ to be an abelian scheme over $S$ with an action $i\colon \cO_{F,(p)}\to\End_S(A)\otimes_\bZ\bZ_{(p)}$.

Then the isomorphism classes of $\cO_{F,(p)}$-linear $\bZ_{(p)}^\times$-isogenies from $A$ to different targets consisting of an abelian scheme over $S$ with an $\cO_{F,(p)}$-action are in bijection with $\pi_1(S,s)$-invariant open compact subgroups of $V^p(A_s)$.
This is by sending the class of $f\in \Hom_{\cO_{F,(p)}}(A,A')\otimes_{\bZ}\bZ_{(p)}$ to $V^p(f)^{-1}(T^p(A'_s))$.

Moreover, the $\cO_{F,(p)}$-action of $A'$ is derived from an action $\cO_F\to \End_S(A')$ if and only if the corresponding subgroup of $V^p(A_s)$ is $\cO_F$-invariant.
\end{lem}

\begin{prop}
The forgetful functor $\cM_{K_{\wt G}}^{\mathrm{naive}}\to\cR_{K_{\wt G}}^{\mathrm{naive}}$ is an equivalence of categories fibered in groupoids.
\end{prop}

\begin{proof}
We show the equivalence of the $S$-valued points, mimicking \cite[Proposition 1.4.3.4]{LanCpt}.
We may assume that $S$ is connected.
We take its geometric point $s$.
We first address the essential surjectivity.
Take an object $(A_0,i_0,\lambda_0,\epsilon^p,A,i,\lambda,\eta^p)$ of $\cR_{K_{\wt G}}^{\mathrm{naive}}$.
By Lemma \ref{lem:CptSub}, there is an abelian scheme $A'$ over $S$, an action $\cO_F\to\End_S(A')$ and an $\cO_{F,(p)}$-linear $\bZ_{(p)}^\times$-isogeny $f\in \Hom(A,A')\otimes_{\bZ}\bZ_{(p)}$ such that
\[
    V^p(f)^{-1} (T^p(A'_s)) = \eta^p(\wh L^p)(T^p(A_{0,s})).
\]
We also obtain a $\bQ$-polarization $\lambda'$ of $A'$ from $\lambda$.
The Kottwitz condition and the sign condition is still valid with $A'$ in place of $A$ and the condition on $\Ker\lambda'_{v}$ holds again.
We are left with $\eta^p$.
We fix a representative $L\otimes_{\bZ}\bA_f^p\isoarrow V^p(A_{0,s},A_s)$ of $\eta^p$ at $s$ denoted by the same symbol.
The homomorphism $V^p(f)\colon V^p(A_{0,s},A_s)\to V^p(A_{0,s},A'_s)$ is isometric due to the compatible $\lambda$ and $\lambda'$.
This leads to the isometry of $V^p(f)\circ\eta^p$.
The image $V^p(f)\circ\eta^p(\wh L^p)$ equals $T^p(A_0,A')$ since $\eta^p$ induces the isomorphism $(L\otimes_\bZ\bA_f^p)\otimes_{\bA_{F,f}^p}V^p(A_{0,s})\isoarrow V^p(A_s)$, through which $\wh L^p\otimes_{\wh{\cO_F}^p}T^p(A_{0,s})$ and $V^p(f)^{-1}(T^p(A'_s))$ coincide.

Next we focus on the fully-faithfulness.
The faithfulness is trivial.
For the other part, take two objects $(A_0,i_0,\lambda_0,\epsilon^p,A,i,\lambda,\eta^p)$ and $(A'_0,i'_0,\lambda'_0,\epsilon'^p,A',i',\lambda',\eta'^p)$ of $\cM_{K_{\wt G}}^{\mathrm{naive}}(S)$, and a morphism of $\cR_{K_{\wt G}}^{\mathrm{naive}}(S)$ from the first to the second.
Let $f\in\Hom_{\cO_F}(A,A')\otimes_\bZ \bZ_{(p)}$ be its constituent.
Then, using the level structures, we have $V^p(f)(T^p(A_s)) = T^p(A'_s)$, implying that $f$ is an actual morphism.
\end{proof}

The functors $\cM_{K_{\wt G}}^{\mathrm{naive}}$ and $\cR_{K_{\wt G}}^{\mathrm{naive}}$ are representable by $\cO_{E,(p)}$-schemes of finite type by a standard argument (\cf\cite[\S 5]{KotInt}).
These schemes are separated by the valuative criterion.
Besides, their generic fibers are $\Sh_{K_{\wt G}}(\wt G, h_{\wt G})$ by \cite[Theorem 4.4]{RSZExp}.

\subsection{Flat models}

When $K_{Z^\bQ}^p = Z^\bQ(\wh\bZ^p)$, we refer to \cite[Theorem 5.4]{RSZExp} for the definition of $\cM_{K_{\wt G}}$, except that it is a scheme over $\cO_{E,(\nu)}$ as in \cite[\S 4]{RSZInt}, and that in (d) of the theorem, we suppose that
$
    \{r_{\varphi_\psi}, r_{\ol{\varphi_\psi}}\} = \{1,n-1\}.
$
The latter is because we do not know if the local model in \cite{SmiPre} related to (d) is normal, affecting in turn the proof of the normality of our models.
The definition of \cite[Theorem 5.4]{RSZExp} uses $\cR_{K_{\wt G}}^{\mathrm{naive}}$, but we can describe it in terms of $\cM_{K_{\wt G}}^{\mathrm{naive}}$.
In general, we put $\cM_{K_{\wt G}} = \cM_{Z^\bQ(\wh{\bZ})K_G}\times_{\cM_0^{\fa,\xi}}\cM_0$.

Our main theorem is the following.

\begin{thm}\label{main}
The system of schemes $\cM_{K_{\wt G}}$ for varying $K_G^p$ and $K_{Z^\bQ}^p$ and the transition morphisms is a canonical integral model with respect to $(\wt\cG, \Mloc)$.
In particular, the system is canonically isomorphic to the system of models in \cite[4.2.1]{KPInt} with the corresponding levels.
\end{thm}

The rest of this section concisely handles the fact that the models satisfy properties in Definition \ref{candef} other than (\ref{extension}) and (\ref{strongdisplay}).
The $\cO_{E,(\nu)}$-scheme $\cM_{K_{\wt G}}$ is flat with generic fiber $\Sh_{K_{\wt G}}(\wt G, h_{\wt G})$ (\cf \cite[Theorem 5.4]{RSZExp}).
It is separated of finite type since so is the naive model.
We claim the normality of $\cM_{K_{\wt G}}$.
The local model diagram in the proof of \cite[Theorem 5.4]{RSZExp} attributes this to the normality of local models.
Local models cited in the same proof are the Pappas--Zhu local models by those citations and \cite[\S 8.2.5(c)]{PZLoc}.
This implies that their base change over $\cO_{E_{\nu}}$ is normal by \cite[Theorem 9.1]{PZLoc} (\cf \cite[Proposition 9.2]{PZLoc}).
The finite etaleness of the transition morphisms is similar to Proposition \ref{auxiliary finet}.

\section{The extension property} \label{sec:Ext}

This section is about the extension property, that is, (\ref{extension}) of Definition \ref{candef}.
First we show the extension property for $\cM_{K_{\wt G}}^{\mathrm{naive}}$.
If $S$ is a locally noetherian $\cO_{E,(p)}$-scheme, then as a set, $\displaystyle\varprojlim_{K_{Z^\bQ}^p, K_G^p}\cM_{K_{\wt G}}^{\mathrm{naive}}(S)$ is the set of isomorphism classes of tuples $(A_0,i_0,\lambda_0,\epsilon^p,A,i,\lambda,\eta^p)$ of
\begin{itemize}
    \item an abelian scheme $A_0$ over $S$,
    \item $i_0\colon \cO_F\to \End_S(A_0)$ with the Kottwitz condition,
    \item an $\cO_F$-linear polarization $\lambda_0\colon A_0\to A_0^\vee$ of degree prime to $p$ such that $(A_0,i_0,\lambda_0)\in\cM_0^{\fa,\xi}(S)$,
    \item an $\wh{\cO_F}^p$-linear isomorphism $\epsilon^p\colon(\wh{\Lambda_0}^p)_S\to T^p(A_0)$ transferring $\psi_0$ to the Riemann form of $\lambda_0$ up to $(\wh\bZ^p)^\times$,
    \item $A$ is an abelian scheme over $S$,
    \item $i\colon\cO_F\to\End_S(A)$ with the other Kottwitz condition,
    \item a $\bQ$-polarization $\lambda\in\Hom_{\cO_F}(A,A^\vee)\otimes_{\bZ}\bZ_{(p)}$, and
    \item a hermitian isomorphism $\eta^p\colon (\wh L^p)_S\to T^p(A_0,A)$,
\end{itemize}
further satisfying two conditions in the definition of $\cM_{K_{\wt G}}^{\mathrm{naive}}$.

Let $R$ be a discrete valuation ring of mixed characteristic $(0,p)$ and let $(A_0,i_0,\lambda_0,\epsilon^p,A,i,\lambda,\eta^p)$ give an element of $\displaystyle\varprojlim_{K_{Z^\bQ}^p, K_G^p}\cM_{K_{\wt G}}^{\mathrm{naive}}\left(R[1/p]\right)$.
Then by the N\'eron--Ogg--Shafarevich criterion, $A_0$ uniquely extends to an abelian scheme $\wt{A_0}$ over $R$.
The data $i_0$ and $\lambda_0$ uniquely extends to data on $\wt{A_0}$ by, e.g., \cite[Lemma 1]{FalLift} and \cite[Proposition 2.14]{MilInt}.
For $\epsilon^p$, the finite etale group scheme $\wt{A_0}[N]$ for a positive integer $N$ prime to $p$ is constant as $A_0[N]$ is, so that $\epsilon^p$ also uniquely extends to a datum over $R$.
The rest of the argument is analogous except that we use $(\wh L^p)_S\otimes_{\wh{\cO_F}^p}T^p(A_0)\isoarrow T^p(A)$ induced by $\eta^p$.

The extension property for $\cM_{K_{\wt G}}$ holds by that for the naive model and the valuative criterion for the closed immersion $\cM_{K_{\wt G}}\to\cM_{K_{\wt G}}^{\mathrm{naive}}$.

\section{The locally universal display} \label{sec:LocUniv}

To prove (\ref{strongdisplay}) of Definition \ref{candef} for RSZ models, we first construct canonical morphisms from them to the Kisin--Pappas models as stated in Theorem \ref{main}.
Next, we pull back the $(\wt\cG, \Mloc)$-displays on the latter.
We are reduced to show the local universality of the new $(\wt\cG, \Mloc)$-displays.
We show this by comparing RSZ models with the full moduli of abelian schemes through the deformation theory.

\subsection{The morphisms to Kisin--Pappas models} \label{sec:LocUniv-MorKP}

We follow \cite[8.1.3]{PapInt} to recall the Kisin--Pappas models.
Put $K^\flat_p = \GSp(\bZ_p)$.
For a compact open subgroup $K^{\flat, p}$ of $\GSp (\bA_f^p)$, set $K^\flat = K^\flat_pK^{\flat, p}$.
By \cite[Lemma 2.1.2]{KisInt}, we can choose $K^{\flat, p}$ so that the Hodge embedding in \S \ref{sec:Def-Emb} induces an embedding $\Sh_{K_{\wt G}}(\wt G, X)\to \Sh_{K^\flat}(\GSp, S^\pm)_E$ over $E$.

Let $\sS_{K^\flat}$ be the quasi-projective $\bZ_{(p)}$-scheme representing the pseudo-functor which carries a locally noetherian $\bZ_{(p)}$-scheme $S$ to the groupoid of triples $(A,\lambda, \eta^p)$, where
\begin{itemize}
    \item $A$ is an abelian scheme over $S$ of relative dimension $[F:\bQ](n+1)$,
    \item $\lambda\in\Hom(A,A^\vee)\otimes_\bZ \bZ_p$ is a $\bQ$-polarization and
    \item $\eta^p$ is a section over $S$ of the pro-etale sheaf
    \[
        \ul{\Isom}(\Lambda_0\otimes_\bZ\bA_f^p\oplus\Lambda\otimes_\bZ\bA_f^p, V^p(A))/K^{\flat, p},
    \]
    the pro-etale sheaf of $\bA_f^p$-linear isomorphisms of the two sheaves transferring $\psi_0 + \psi$ to the Riemann form of $\lambda_0$ up to a constant in $\bA_f^{p,\times}$.
\end{itemize}
The morphisms from $(A,\lambda, \eta^p)$ to $(A',\lambda', \eta'^p)$ in $\sS_{K^\flat}(S)$ are prime-to-$p$ isogenies from $A$ to $A'$ carrying $\lambda$ to $\lambda'$ up to a constant in $\bZ_p^\times$ and preserving the level structures.

We explain the morphism $\cM_{K_{\wt G}}\to\sS_{K_{\wt G}}$.
Suppose $S$ is a connected locally noetherian $\cO_{E,(\nu)}$-scheme with a geometric point $s\to S$ and $(A_0,i_0,\lambda_0,\epsilon^p,A,i,\lambda,\eta^p)$ is an element of $\cM_{K_{\wt G}}(S)$.
Then, $\epsilon^p$ and $\eta^p$ is given by a $K_{Z^\bQ}^p$-class of $\wh{\cO_F}^p$-linear isomorphisms $\wh{\Lambda_0}^p\isoarrow T^p(A_{0,s})$ and $\wh L^p\isoarrow T^p(A_{0,s}, A_s)$.
These combine to give $\wh{\Lambda_0}^p\oplus \wh\Lambda^p\isoarrow T^p(A_{0,s}\times_s A_s)$, forming an element of $\sS_{K^\flat}(S)$ with $A_0\times_S A$ and $\lambda_0\times_S \lambda$.
This defines a forgetful morphism $\cM_{K_{\wt G}}\to\sS_{K_{\wt G}}$.

\subsection{The local universality}

Let $\cL_{K_{\wt G}}$ be the pro-etale $\wt\cG(\bZ_p)$-local system on $\Sh_{K_{\wt G}}(\wt G,X)$ given by the covers
\[
    (\Sh_{K'_pK_{\wt G}^p}(\wt G,X)\to\Sh_{\wt\cG(\bZ_p)K_{\wt G}^p}(\wt G,X))_{K'_p\subseteq\wt\cG(\bZ_p)}.
\]
Let $\sD_{K_{\wt G}}$ be the pull-back on $\cM_{K_{\wt G}}$ of the locally universal $(\wt\cG, \Mloc)$-display on $\sS_{K_{\wt G}}$ associated with $\cL_{K_{\wt G}}$.
This pullback is a $(\wt\cG, \Mloc)$-display associated with $\sD_{K_{\wt G}}$ that is compatible for different levels.

To show its local universality, we now discuss the deformation theory.
Set $\breve{E_\nu}$ to be the maximal unramified extension of $E_\nu$.
The symbol $\cC$ denotes the category of artinian local $\cO_{\breve{E_\nu}}$-algebras whose structure homomorphism is local and induces the equality of residue fields.
Take $x\in\cM_{K_{\wt G}}(\ol{\bF_p})$.
Let $z\in \sS_{K^\flat}(\ol{\bF_p})$ be its image.
$R_x$ (resp. $R_z$) denotes the strict completion of the local ring of $\sS_{K_{\wt G}}$ at $x$ (resp. $\sS_{K^\flat}\otimes_{\bZ_{(p)}}\cO_{E, (\nu)}$ at $z$).
Let $(\ol A_0,\ol i_0,\ol\lambda_0,\ol\epsilon^p,\ol A,\ol i,\ol\lambda, \ol\eta^p)$ be the object of $\cM_{K_{\wt G}}(\ol{\bF_p})$ in correspondence with $x$.
Let $(\ol A_0\times_{\bF_p}\ol A, \ol\lambda_0\times_{\bF_p}\ol\lambda, \ol{\zeta}^p)$ in $\sS_{K^\flat}(\ol{\bF_p})$ correspond with $z$.

The complete noetherian ring $R_x$ prorepresents the following set-valued deformation functor $\Def(x)$ from $\cC$.
For any object $S$ of $\cC$, the set $\Def(x)(S)$ comprises the isomorphism classes of the objects of $\cM_{K_{\wt G}}(S)$ reduced to $(\ol A_0,\ol i_0,\ol\lambda_0,\ol\epsilon^p,\ol A,\ol i,\ol\lambda, \ol\eta^p)$ over $\ol{\bF_p}$.
Likewise, $R_z$ prorepresents the set-valued deformation functor $\Def(z)$ from $\cC$ carrying an object $S$ of $\cC$ to the set of equivalence classes in $\sS_{K^\flat}(S)$ reduced to the class of $\ol A_0\times_{\bF_p}\ol A, \ol\lambda_0\times_{\bF_p}\ol\lambda, \ol{\zeta}^p$ over $\ol{\bF_p}$.

\begin{lem}\label{climm}
The forgetful morphism $\Def(x)\to\Def(z)$ is a closed immersion.
\end{lem}
\begin{proof}
First, we have a closed immersion
\[
    \Def(x)\to\Def(x)^{\mathrm{naive}}
\]
to a deformation functor for $\cM_{K_{\wt G}}^{\mathrm{naive}}$ similarly defined to the functor $\Def(x)$.
Next, we may forget about the level structures both in $\Def(x)^{\mathrm{naive}}$ and $\Def(z)$ since they are described in terms of etale fundamental groups and stalks of etale sheaves, always admitting a unique deformation when the other data are deformed.
We may also forget the locally constant sign invariant in $\Def(x)^{\mathrm{naive}}$.
Then, forgetting the closed conditions on $\Ker\lambda_0$ and $\Ker\lambda_v$ and the Kottwitz conditions and using the Serre--Tate theorem yields a closed immersion
\[
    \Def(x)^{\mathrm{naive}}\to\Def(\ol A_0[p^\infty],\ol\lambda_0,\ol i_0)\times\Def(\ol A[p^\infty],\ol\lambda,\ol i),
\]
the target consisting of the deformation functors on $\cC$ of the data in the parentheses.
Moreover, we have a closed immersion \cite[Theorem 1.4.5.5]{CCOLift}
\[
    \Def(\ol A_0[p^\infty],\ol\lambda_0,\ol i_0)\times\Def(\ol A[p^\infty],\ol\lambda,\ol i)\to\Def(\ol A_0[p^\infty], \ol\lambda_0)\times\Def(\ol A[p^\infty], \ol\lambda).
\]
Furthermore, we have a closed immersion
\[
    \Def(\ol A_0[p^\infty], \ol\lambda_0)\times\Def(\ol A[p^\infty], \ol\lambda)\to \Def(\ol A_0[p^\infty]\times_{\bF_p}\ol A[p^\infty], \ol\lambda_0\times_{\bF_p}\ol\lambda).
\]
The latter is isomorphic to $\Def(z)$ again by the Serre--Tate theorem.
All these are composed to give the desired result.
\end{proof}

\begin{cor}
The $(\wt\cG, \Mloc)$-display $\sD_{K_{\wt G}}$ is locally universal.
\end{cor}
\begin{proof}
We just have to show that for the image $y\in\sS_{K_{\wt G}}(\ol{\bF_p})$ of $x$, the map $R_y\to R_x$ between the strict completions of the local rings is isomorphic due to the local universality for $\sS_{K_{\wt G}}$.
Showing that it is an isomorphism is further attributed to proving that $R_z$ surjects on $R_x$ since $R_x$ and $R_y$ has the same dimension by the local model diagrams in \cite[Theorem 4.2.7]{KPInt} and \cite[Theorem 5.4]{RSZExp}.
The proof is completed by Lemma \ref{climm}.
\end{proof}
Thus RSZ models form a canonical system.
The morphisms from RSZ models to Kisin--Pappas models are isomorphic just as in the proof of \cite[Theorem 7.1.7]{PapInt}.
This completes the proof of Theorem \ref{main}.

\bibliographystyle{amsplain}

\noindent
Yuta Nakayama\\
Graduate School of Mathematical Sciences, The University of Tokyo,
3-8-1 Komaba, Meguro-ku, Tokyo, 153-8914, Japan \\
nkym@ms.u-tokyo.ac.jp

\end{document}